\documentclass[envcountsame, envcountsect]{svmult}
\usepackage{etex,amssymb}
\usepackage[unicode]{hyperref}
\usepackage{mathtools}
\usepackage{xcolor}
\usepackage[color,matrix,arrow, curve]{xy}
\usepackage{tkz-graph}

\usepackage{tikz}
\usetikzlibrary{arrows,decorations.pathmorphing}
\usetikzlibrary{backgrounds,positioning}
\usetikzlibrary{fit,petri,shapes.misc}
\usepackage{tikz-cd}

\tikzset{auto}

\tikzset{empty/.style={circle,inner sep=0pt,minimum size=6mm}}
\tikzset{emptyvt/.style={circle,inner sep=0pt,minimum size=0mm}}

\tikzset{grow'=down}

\tikzset{plain/.style={circle,draw,very thick,
inner sep=0pt,minimum size=6mm}}

\tikzset{fatplain/.style={rounded rectangle,draw,very thick,minimum size=6mm}}

\tikzset{bigplain/.style={rounded rectangle,draw,very thick,minimum size=.8cm}}

\tikzset{yellowvt/.style={circle,draw,fill=yellow,very thick,inner sep=0pt,minimum size=6mm}}

\tikzset{bluevt/.style={circle,draw,fill=blue!20,very thick,inner sep=0pt,minimum size=6mm}}

\tikzset{greenvt/.style={circle,draw,fill=green!30,very thick,inner sep=0pt,minimum size=6mm}}

\tikzset{redvt/.style={circle,draw,fill=red!30,very thick,inner sep=0pt,minimum size=6mm}}

\tikzset{arrow/.style={->,thick}}
\tikzset{dashedarrow/.style={->,dashed,thick}}
\tikzset{dottedarrow/.style={->,dotted,thick}}
\tikzset{mapto/.style={|->,thick}}

\tikzset{implies/.style={thick,double,double equal sign distance,-implies}}

\tikzset{line/.style={thick}}
\tikzset{dottedline/.style={dotted,thick}}
\tikzset{dashedline/.style={dashed,thick}}

\tikzset{inputleg/.style={<- thick}}
\tikzset{outputleg/.style={->,thick}}
\tikzset{dottedinput/.style={<-,dotted,thick}}

\GraphInit[vstyle = Shade]
\tikzset{
  LabelStyle/.style = { rectangle, rounded corners, draw,
                        minimum width = 2em, fill = yellow!50,
                        text = red, font = \bfseries },
  VertexStyle/.append style = { inner sep=5pt,
                                font = \Large\bfseries},
  EdgeStyle/.append style = {->, bend left} }
\usepackage[normalem]{ulem}
\usepackage{rotating}
\DeclareFontFamily{U}{wncy}{}
\DeclareFontShape{U}{wncy}{m}{n}{<->wncyr10}{}
\DeclareSymbolFont{mcy}{U}{wncy}{m}{n}
\DeclareTextSymbol{\shaa}{mcy}{58}
\DeclareMathSymbol{\Sha}{\mathord}{mcy}{"58}

\DeclareMathOperator{\Aut}{Aut}

\DeclareMathOperator{\col}{col}
\DeclareMathOperator{\Coll}{Col}
\DeclareMathOperator*{\colim}{colim}

\DeclareMathOperator{\Ed}{Edge}

\newcommand{\Graph}{\mathsf{Graph}}

\DeclareMathOperator{\inn}{in}
\DeclareMathOperator{\Iso}{Iso}

\DeclareMathOperator{\map}{map}

\DeclareMathOperator{\Ob}{Ob}

\DeclareMathOperator{\op}{op}
\newcommand{\Operad}{\mathsf{Operad}}

\DeclareMathOperator{\out}{out}

\DeclareMathOperator{\Properad}{Properad}

\DeclareMathOperator{\Sc}{Sc}

\DeclareMathOperator{\Set}{Set}

\DeclareMathOperator{\sSet}{sSet}

\DeclareMathOperator{\Vt}{Vt}

\newcommand{\boundary}{\partial}

\newcommand{\cCc}{\mathfrak{C}}

\newcommand{\Mm}{\mathcal{M}}
\newcommand{\nN}{\mathbb{N}}

\newcommand{\rR}{\mathbb{R}}

\newcommand{\ul}[1]{\underline{#1}}

\newcommand{\im}{\mathrm{im}}
\newcommand{\id}{\mathrm{id}}

\spnewtheorem{nondefinition}[theorem]{Bad Definition}{\bfseries}{\rmfamily}

\smartqed

\title*{Lecture notes on infinity-properads}
\author{Philip Hackney and Marcy Robertson}
\institute{Philip Hackney 
\at Universit\"at Osnabr\"uck, \email{phackney@uos.de} 
\and
Marcy Robertson 
\at The University of Melbourne, \email{marcy.robertson@unimelb.edu.au}}
\begin{document}

\abstract*{These are notes for three lectures on higher properads given at a program at the mathematical institute MATRIX in Australia in June 2016. 
The first lecture covers the case of operads, and provides a brief introduction to the Moerdijk-Weiss theory of dendroidal sets. 
The second lecture extends the discussion to properads and our work with Donald Yau on graphical sets. 
These two lectures conclude with models for higher (pr)operads given by an inner horn filling condition. 
Finally, in the last lecture, we explore some properties of the graphical category and use them to propose a Segal-type model for higher properads.}

\maketitle

The main goal of this lecture series is to provide a brief introduction to the theory of higher operads and properads.
As these informal lecture notes stay very close to our presentations, which occupied only three hours in total, we were necessarily extremely selective in what is included.
It is important to reiterate that this is \emph{not} a survey paper on this area, and the reader will necessarily have to use other sources to get a `big picture' overview.

Various models of infinity-operads have been developed in work of C.\,Barwick,  
D.-C.\,Cisinski, J.\,Lurie, I.\,Moerdijk, I.\,Weiss and others~\cite{1302.5756, cm-ho,cm-ds, cm-simpop, higheralgebra, mw,mw2}. In these lectures we focus on the combinatorial models which arise when one extends the simplicial category $\varDelta$ by a category of trees $\varOmega$. 
This `dendroidal category' leads immediately to the category of dendroidal sets \cite{mw}, namely the presheaf category $\Set^{\varOmega^{op}}$.
A dendroidal set $X\in\Set^{\varOmega^{op}}$ which satisfies an inner horn-filling condition is called a quasi-operad (see Definition \ref{quasi-operad definition}).
We briefly review these objects in section \ref{Lecture1}.
 
Properads are a generalization of operads introduced by B.\,Vallette~\cite{vallette} which parametrize algebraic structures with several inputs and several outputs. These types of algebraic structures include Hopf algebras, Frobenius algebras and Lie bialgebras.  In our monograph \cite{hrybook} with D.\,Yau and in subsequent papers, we work to generalize the theory of infinity-operads to the properad setting. In section \ref{Lecture2} we explain the appropriate replacement of the dendroidal category $\varOmega$ the graphical category $\varGamma$ and define quasi-properads as graphical sets which satisfy an inner horn-filling condition.
This material (and much more) can be found in the monograph \cite{hrybook}.
It is worth mentioning that J.\,Kock, while reading the  manuscript of \cite{hrybook}, realized that one can give an alternative definition of the category $\varGamma$.
The interested reader can find more details of this construction in \cite{kock}. 

In the final section, we propose a Segal-type model for infinity properads.
There are clear antecedents for models of this form in several other settings \cite{9807049,threemodels,cm-ds,bh1}.
We recall the C.\,Berger and I.\,Moerdijk theory of generalized Reedy categories from \cite{bm_reedy}.
The graphical category $\varGamma$ is such a category, so the category of graphical spaces $\sSet^{\varGamma^{op}}$ possesses a cofibrantly generated model structure with levelwise weak equivalences and relatively few fibrant objects.
Finally, we discuss the Segal condition in the context of graphical sets and spaces.

\subsection*{Acknowledgments} These lectures were given in the inaugural workshop at the mathematical research institute MATRIX in Australia called ``Higher Structures in Geometry and Physics'' in June 2016; needless to say, these notes would not exist had MATRIX not supported us and allowed us to host the program in the first place. 
We would like to thank all the participants of the workshop for asking interesting questions and forcing us to refine these ideas, and also to Jon Beardsley, Julie Bergner, and Joachim Kock for offering feedback on earlier drafts of these notes.  
A special thank you goes to Gabriel C.\,Drummond-Cole who generously shared his live\TeX{ed} notes which formed the backbone of this document. We are also grateful to the Hausdorff Research Institute for Mathematics and the Max Planck Institute for Mathematics for their hospitality while we were finishing the writing and editing of these notes.

\section{Colored operads, dendroidal sets, and quasi-operads}\label{Lecture1}

This section is a brief overview of dendroidal sets,  introduced by Moerdijk and Weiss~\cite{mw}, which allow us to discuss the `quasi-operad' model for infinity categories \cite{mw,cm-ho}. Throughout this section, we are using the formal language that we will need to extend to the more subtle case of properads. For those who are unfamiliar with dendroidal sets we recommend the original paper \cite{mw} and the lecture notes by Moerdijk~\cite{moerdijklecture} as references. 

\begin{definition}\label{def_graph}
A \emph{graph} is a connected, directed graph $G$ which admits legs and does not admit directed cycles. 
A \emph{leg} is an edge attached to a vertex at only one end.
 We also want our graphs to have an ordering given by bijections 
 $$\textrm{ord}^{in}_{G}:\{1,...,m\}\longrightarrow \inn(G)$$
 $$\textrm{ord}^{out}_{G}:\{1,...,n\}\longrightarrow \out(G)$$ as well as bijections 
$$\textrm{ord}^{in}_{v}:\{1,...,k\}\longrightarrow \inn(v)$$ and 
$$\textrm{ord}^{out}_{v}:\{1,...,j\}\longrightarrow \out(v)$$ for each $v$ in $\Vt(G)$.

If we say that $G$ is a $\cCc$-colored graph then we are including the extra data of an edge coloring function $\eta:\Ed(G)\longrightarrow \cCc.$ 

\end{definition} 
When we draw pictures of graphs, we will omit the arrows, and always assume the direction in the direction of gravity. 

\begin{definition}
A \emph{tree} is a simply connected graph with a unique output (the root).
\end{definition}

For any vertex $v$ in a $\cCc$-colored tree $T$ $\inn(v)$ is written as a list $\ul{c} = c_1,\ldots, c_k$ of colors $c_i\in \cCc$. A list of colors like $\ul{c}$ is called a \emph{profile} of the vertex $v$. Similarly,  $\out(v)=d$ identifies the element $d\in\cCc$ which colors the output of the vertex $v$. The complete input-output data of a vertex $v$ is given by the \emph{biprofile} $(\ul{c};d)$. 

\begin{example} 
In the following picture the tree has legs labeled $3,4,5, 6$ and $0$. The leg $0$ is the single output of this graph.  Internal edges are labeled $1,2$ and $7$. The edges at each vertex all come equipped with an ordering, and if we wish to list the inputs to the vertex $v$ we would write $\inn(v) =(1,2)$. 

\[
\begin{tikzpicture} [scale=.7]
\node[empty](k) at (-1,2){};
\node[empty](l) at (0,2){};
\node[empty](j) at (1,2){}; 
\node[empty](p) at (2,2){}; 
\node[shape=circle,draw=black] (q) at (4,2){q}; 
\node[shape=circle,draw=black] (w) at (0,0) {w};
\node[shape=circle, draw=black](v) at (1,-2) {v};
\node[shape=circle, draw=black](y) at (3,0) {y};
\node[empty](m) at (1,-4){};

\draw[-](y) to node [midway, fill=white] {2} (v);
\draw[-] (w) to node [midway, fill=white] {1} (v);
\draw[-] (k) to node [above, fill=white] {3} (w);
\draw[-] (l) to node [above, fill=white] {4} (w);
\draw[-] (j) to node [above, fill=white] {5} (w);
\draw[-](v) to node [midway, fill=white] {0} (m);
\draw[-](p) to node [above, fill=white] {6}(y);
\draw[-](q) to node [midway, fill=white] {7} (y);

\end{tikzpicture}
\]
If we wanted to consider $T$ as a $\cCc$-colored tree, we would add the data of a coloring function $\Ed(T)\rightarrow \cCc$ which would result in our picture looking like 
\[
\begin{tikzpicture} [scale=.7]
\node[empty](k) at (-1,2){};
\node[empty](l) at (0,2){};
\node[empty](j) at (1,2){}; 
\node[empty](p) at (2,2){}; 
\node[shape=circle,draw=black] (q) at (4,2){q}; 
\node[shape=circle,draw=black] (w) at (0,0) {w};
\node[shape=circle, draw=black](v) at (1,-2) {v};
\node[shape=circle, draw=black](y) at (3,0) {y};
\node[empty](m) at (1,-4){};

\draw[-](y) to node [midway, fill=white] {$c_2$} (v);
\draw[-] (w) to node [midway, fill=white] {$c_1$} (v);
\draw[-] (k) to node [above, fill=white] {$c_3$} (w);
\draw[-] (l) to node [above, fill=white] {$c_4$} (w);
\draw[-] (j) to node [above, fill=white] {$c_5$} (w);
\draw[-](v) to node [midway, fill=white] {$d$} (m);
\draw[-](p) to node [above, fill=white] {$c_6$}(y);
\draw[-](q) to node [midway, fill=white] {$c_7$} (y);

\end{tikzpicture}
\]
where $d$ and each of the $c_i$ are elements of $\mathfrak C$.
The profile of $v$ is $\inn(v)=(c_1,c_2)=\ul{c}$ and the biprofile of $v$ is written as $(\ul{c};d)$, where the semi-colon differentiates between inputs and outputs. 

\end{example}

\subsection{Colored operads}

A \emph{colored operad} is a generalization of a category in which we have a set of objects (or colors) but where we allow for morphisms which have a finite list of inputs and a single output.  When we visualize these morphisms we write them as colored trees, so that the morphism  
$
\xymatrix{
(x_1,x_2) \ar[r]^-{f} & y
}$ looks like

 \[
\begin{tikzpicture} [scale=.7]
\node[empty](k) at (-1,2){};
\node[empty](j) at (1,2){}; 
\node[shape=circle,draw=black] (f) at (0,0) {f};
\node[empty](v) at (0,-2) {};
\draw[-] (f) to node [midway, fill=white] {$y$} (v);
\draw[-] (k) to node [above, fill=white] {$x_1$} (f);
\draw[-] (j) to node [above, fill=white] {$x_2$} (f);

\end{tikzpicture}
\] 
Notice that in this depiction the edges of the tree are colored by the objects (hence the name colors). 
A modern comprehensive treatment of colored operads appears in the book of D.\,Yau \cite{yauoperad}.

\begin{definition}
A \emph{colored operad} $P$ consists of the following data:
\begin{enumerate}
\item A set of colors $\cCc=\col(P)$; 
\item for all $n\geq 0$ and all biprofiles $(\ul{c};d) =(c_1,...,c_n; d)$ in $\cCc$, a set $P(\ul{c};d)$;
\item for $\sigma \in \varSigma_n$, maps
$\sigma^* \colon P(\ul{c}; d) \to P(\ul{c}\sigma; d) = P(c_{\sigma(1)}, \dots, c_{\sigma(n)}; d)
$ so that $(\sigma \tau)^* = \tau^*\sigma^*$;
\item for each $c\in\cCc$ a unit element $\id_c \in P(c;c)$;
\item associative, equivariant and unital compositions $$P(\ul{c};d)\circ_{i}P(\ul{d};c_i)\rightarrow P(c_1,...,c_{i-1},(d_1...,d_k),c_{i+1},...,c_m; d)$$ where $\ul{d}=(d_1,...,d_k)$ and $1\leq i\leq m$. 
\end{enumerate}
A morphism $f:P\rightarrow Q$ consists of: 
\begin{enumerate} 
\item a map of color sets $f:\col(P)\rightarrow \col(Q)$;
\item for all $n\ge 0$ and all biprofiles $(\ul{c};d)$, a map of sets $$f:P(\ul{c},d)\rightarrow Q(f\ul{c},fd)$$ which commutes with symmetric group actions, composition and units. 
\end{enumerate} 
The category of colored operads is denoted by $\Operad$. 
\end{definition}

Examples of colored operads include: 

\begin{itemize} 
\item The $2$-colored operad $\mathsf{O}^{[1]}$, whose algebras are morphisms of $\mathsf{O}$-algebras for a specified uncolored operad $\mathsf{O}$ \cite[1.5.3]{bmresolution}
 
\item The $\mathbb{N}$-colored operad whose algebras are all one colored operads \cite[1.5.6]{bmresolution}, \cite[\S 14.1]{yj}, \cite{yauoperad}. 
\end{itemize} 

We now focus on operads which are generated by uncolored trees.  Explicitly, given any uncolored tree $T$, one can generate a colored operad $\varOmega(T)$ so that
\begin{itemize} 
\item  the set of colors of $\varOmega(T)$ is taken to be the set of edges of $T$; 
\item the operations of $\varOmega(T)$ are freely generated by vertices in the tree.
\end{itemize} 

\begin{example} Consider the uncolored tree $T$ 
\[
\begin{tikzpicture} [scale=.7]
\node[empty](k) at (-1,2){};
\node[empty](l) at (0,2){};
\node[empty](j) at (1,2){}; 
\node[empty](p) at (2,2){}; 
\node[shape=circle,draw=black] (q) at (4,2){q}; 
\node[shape=circle,draw=black] (w) at (0,0) {w};
\node[shape=circle, draw=black](v) at (1,-2) {v};
\node[shape=circle, draw=black](y) at (3,0) {y};
\node[empty](m) at (1,-4){};

\draw[-](y) to node [midway, fill=white] {a} (v);
\draw[-] (w) to node [midway, fill=white] {b} (v);
\draw[-] (k) to node [above, fill=white] {c} (w);
\draw[-] (l) to node [above, fill=white] {d} (w);
\draw[-] (j) to node [above, fill=white] {e} (w);
\draw[-](v) to node [midway, fill=white] {r} (m);
\draw[-](p) to node [above, fill=white] {f}(y);
\draw[-](q) to node [midway, fill=white] {g} (y);

\end{tikzpicture}
\] where we have labeled the edges by letters, but do not mean there is a coloring. 
 
The associated colored operad $\varOmega(T)$  will have color set $$\cCc=\{a,b,c,d,e,f,g,r\}=\Ed(T)$$ and operations freely generated by the vertices. In this example, generating operations are $v\in\varOmega(T)(a,b;r)$, $y\in\varOmega(T)(f,g;a)$, $w\in\varOmega(T)(c,d,e;b)$ and $q\in\varOmega(T)(-;g)$. Composition of operations are given by formal graph substitutions (see Definition~\ref{def_graph_sub}) into appropriate partially grafted corollas (Definition \ref{pgc def}). To give a specific example, the operation $v\circ_{a}y\in\varOmega(T)(b,f,g;r)$ is a composition of $v$ and $y$ which we visualize as being the result of collapsing along the edge marked $a$. 

\end{example}

\begin{definition}\cite{mw} The dendroidal category $\varOmega$ is the full subcategory of $\Operad$ whose objects are colored operads of the form $\varOmega(T)$. When no confusion can arise, we often write $T$ for $\varOmega(T)$.  
\end{definition}

\begin{definition}\cite[Definition 4.1]{mw}
A \emph{dendroidal set} is a functor $X:\varOmega^{op}\to \Set$.
Collectively these form a category $\Set^{\varOmega^{op}}$ of dendroidal sets.
\end{definition}

An element of $x\in X_T$ is called a \emph{dendrex} of shape $T$. We also have the representable functors $\varOmega[T]=\varOmega(-, T).$ 

\subsection{Coface maps and graph substitution}
Quasi-operads are similar in spirit to quasi-categories. In particular, they are dendroidal (rather than simplicial) sets which satisfy an inner Kan condition.  This requires that we define coface and codegeneracy maps in $\varOmega$ which we will make precise by using formal graph substitution. 

\begin{definition}\cite[2.16]{hrybook}\label{pgc def} A \emph{partially grafted corolla} $P$ is a graph with two vertices $u$ and $v$ in which a nonempty finite list of outputs of $u$ are inputs of $v$. 

\end{definition}

\begin{example} The following graph $P$ is a partially grafted corolla. 
\[
\begin{tikzpicture}[scale=.6]
\node[empty](k) at (-1,2){};
\node[empty](l) at (0,2){};
\node[empty](j) at (1,2){}; 
\node[shape=circle,draw=black] (u) at (0,0) {u};
\node[shape=circle, draw=black](v) at (0,-2) {v};
\node[empty](o) at (1,-2){};
\node[empty](p) at (-1,-2){};
\node[empty](m) at (-1,-4){};
\node[empty](n) at (1,-4){};
\node[empty](q) at (1, 1){};

\draw [-] (k) to [bend right =10] (u);
\draw [-] (j) to [bend left=10]  (u);
\draw [-] (l) to (u);

\draw [-] (u) to [bend left= 30]  (v);
\draw[-] (u) to [bend right=30] (v);
\draw[-] (u) to  (v);

\draw [-] (v) to[bend right=10]  (m);
\draw [-] (v) to [bend left=10] (n);
\draw [-] (u) to [bend left=30] (o);
\draw [-] (u) to [bend right=30] (p);
\draw [-] (q) to [bend left=30] (v);
\end{tikzpicture}
\]

\end{example} 

Partially grafted corollas play a key role in describing operadic and properadic composition as it arises from graph substitution. Graph substitution is a formal language for saying something very intuitive, namely that in a given graph G, you can drill a little hole at any vertex and plug in a graph $H$ and assemble to get a new graph. 

\begin{definition}\label{def_graph_sub}\cite[2.4]{hrybook}
We can substitute a graph $H$ into a graph $G$ at vertex $v$ if:
\begin{enumerate}
\item there is a specified bijection $\inn(H)\xrightarrow{\cong}\inn(v)$,
\item a specified bijection $\out(H)\xrightarrow{\cong} \out(v)$, and
\item the coloring of inputs and outputs of $H$ matches the local coloring of $G$ at the vertex $v$.  
\end{enumerate}
The resulting graph is denoted as $G(H_v)$ and we say that $G(H_v)$ was obtained from $G$ via graph substitution. The subscript on $H_v$ indicates that we substituted $H$ into vertex $v$. 
If $S\subseteq \Vt(G)$, we will write $G\{H_v\}_{v\in S}$ when we perform graph substitution at several vertices simultaneously. 
\end{definition}

Graph substitution induces maps in $\varOmega$. For example consider $T$ 

\[
\begin{tikzpicture} [scale=.5]
\node[empty](k) at (-1,2){};
\node[empty](l) at (0,2){};
\node[empty](j) at (1,2){}; 
\node[empty](p) at (2,2){}; 
\node[shape=circle,draw=white] (G) at (-2,0){$T=$}; 
\node[shape=circle,draw=black] (q) at (4,2){q}; 
\node[shape=circle,draw=black] (w) at (0,0) {w};
\node[shape=circle, draw=black](x) at (1,-2) {x};
\node[shape=circle, draw=black](y) at (3,0) {y};
\node[empty](m) at (1,-4){};

\draw[-](y) to (x);
\draw[-] (w) to (x);
\draw[-] (k) to (w);
\draw[-] (l) to (w);
\draw[-] (j) to (w);
\draw[-](x) to (m);
\draw[-](p) to (y);
\draw[-](q) to (y);

\end{tikzpicture}
\]
and the partially grafted corolla $P$ 

\[
\begin{tikzpicture} [scale=.5]
\node[empty](k) at (-1,2){};
\node[empty](l) at (0,2){};
\node[empty](j) at (1,2){}; 
\node[shape=circle,draw=black] (u) at (0,0) {u};
\node[shape=circle,draw=white] (H) at (-2,0) {$P=$};
\node[shape=circle, draw=black](v) at (0,-2) {v};
\node[empty](m) at (0,-4){};

\draw[-] (u) to (v);
\draw[-] (k) to (u);
\draw[-] (l) to (u);
\draw[-] (j) to (u);
\draw[-](v) to (m);

\end{tikzpicture}. 
\]
Since the total number of inputs of $P$ matches the total number of inputs of the vertex $w\in\Vt(G)$ and the number of outputs of $P$ matches the number of outputs of $w$ we can preform graph substitution.

\[
\begin{tikzpicture} [scale=.5]
\node[empty](k) at (-1,4){};
\node[empty](l) at (0,4){};
\node[empty](j) at (1,4){}; 
\node[empty](p) at (2,2){}; 
\node[shape=circle,draw=black] (q) at (4,2){q}; 
\node[shape=circle,draw=black] (v) at (0,0) {v};
\node[shape=circle,draw=black] (u) at (0,2) {u};
\node[shape=circle,draw=white] (GP) at (-3,0) {$T(P_w)=$};
\node[shape=circle, draw=black](x) at (1,-2) {x};
\node[shape=circle, draw=black](y) at (3,0) {y};
\node[empty](m) at (1,-4){};

\draw[-](y) to (x);
\draw[-] (w) to (x);
\draw[-] (k) to (u);
\draw[-] (l) to (u);
\draw[-] (j) to (u);
\draw[-] (u) to (v);
\draw[-](x) to (m);
\draw[-](p) to (y);
\draw[-](q) to (y);

\end{tikzpicture}
\]

Graph substitution induces a map $T\rightarrow T(P_w)$ in $\varOmega$  which sends the $w$ to $u\circ v$, $x$ to $x$, $y$ to $y$, and $q$ to $q$.  This example generalizes, in that if we take any tree $S$ we can expand a vertex to create an additional internal edge by substitution of the proper partially grafted corolla. The expansion of an internal edge can be written as an internal graph substitution, and we have an induced $\varOmega$-map $d^{uv}:S\to T=S(P)$ where $P$ is the appropriate partially grafted corolla. Maps of the type $d^{uv}$ are called \emph{inner coface} maps (\cite[pg. 6]{mw}, \cite[6.1.1]{hrybook}).

Let's look at another example of graph substitution.  Consider the partially grafted corolla $P$

\[
\begin{tikzpicture} [scale=.5]
\node[empty](k) at (-1,2){};
\node[empty](l) at (0,2){};
\node[empty](j) at (1,2){}; 
\node[empty](n) at (1,0){};
\node[empty](o) at (2,-1){};
\node[shape=circle,draw=white] (P) at (-2,0) {$P=$};
\node[shape=circle,draw=black] (u) at (0,0) {u};
\node[shape=circle, draw=black](v) at (0,-2) {v};
\node[empty](m) at (0,-4){};

\draw[-] (u) to (v);
\draw[-] (k) to (u);
\draw[-] (l) to (u);
\draw[-] (j) to (u);
\draw[-](n) to (v);
\draw[-](o) to (v);
\draw[-](v) to (m);

\end{tikzpicture} 
\] and the tree $S$

\[
\begin{tikzpicture} [scale=.5]
\node[empty](p) at (2,2){}; 
\node[empty] (q) at (4,2){}; 
\node[shape=circle,draw=white] (S) at (-2,0) {$S=$};
\node[shape=circle,draw=black] (z) at (0,0) {z};
\node[empty] (u) at (0,2) {};
\node[shape=circle, draw=black](x) at (1,-2) {x};
\node[shape=circle, draw=black](y) at (3,0) {y};
\node[empty](m) at (1,-4){};

\draw[-](y) to (x);
\draw[-] (w) to (x);
\draw[-] (u) to (z);
\draw[-](x) to (m);
\draw[-](p) to (y);
\draw[-](q) to (y);

\end{tikzpicture}. 
\] 

We can substitute $S$ into the vertex $v$ in the partial grafted corolla $P$ since $S$ has the same number of inputs and outputs as $v$.  The resulting picture is the tree $P(S_v)$ 
\[
\begin{tikzpicture} [scale=.5]
\node[empty](k) at (-1,4){};
\node[empty](l) at (0,4){};
\node[empty](j) at (1,4){}; 
\node[empty](p) at (2,2){}; 
\node[empty] (q) at (4,2){}; 
\node[shape=circle,draw=black] (z) at (0,0) {z};
\node[shape=circle,draw=black] (u) at (0,2) {u};
\node[shape=circle, draw=black](x) at (1,-2) {x};
\node[shape=circle, draw=black](y) at (3,0) {y};
\node[empty](m) at (1,-4){};

\draw[-](y) to (x);
\draw[-] (w) to (x);
\draw[-] (k) to (u);
\draw[-] (l) to (u);
\draw[-] (j) to (u);
\draw[-] (u) to (z);
\draw[-](x) to (m);
\draw[-](p) to (y);
\draw[-](q) to (y);

\end{tikzpicture}
\]
and there is a natural map $d^{u}: S\rightarrow P(S_v)$ which is an inclusion of $S$ as a subtree in $P(S_v)$. For any tree $T$ we can write all subtree inclusions by (possibly iterated) substitution of the subtree into a partially grafted corolla and maybe relabeling (\cite[Definition 6.32]{hrybook}).   Maps like these which are induced by graph substitution where the partially grafted corolla is on the ``outside'' are called \emph{outer coface maps} (\cite[pg. 6]{mw}, \cite[6.1.2]{hrybook}).  The third class of maps we will concern ourselves with are called \emph{codegeneracies} and are given by the substitution of a graph with no vertices $\downarrow$ into a bivalent vertex $v$, i.e. the maps $\sigma^v:H\to H(\downarrow)$ (\cite[pg. 6]{mw}, \cite[6.1.3]{hrybook}). The cofaces and codegeneracies satisfy identities reminiscent of the simplicial identities. 

\begin{lemma}\cite[Lemma 3.1]{mw}
The category $\varOmega$ is generated by the inner and outer coface maps, codegeneracies and isomorphisms.
\end{lemma}

In other words, any every map in $\varOmega$ can be factored as a composition of inner and outer coface maps, codegeneracies and isomorphisms. These factorizations will be more carefully discussed in \S\ref{Lecture3}.

\subsection{Boundaries and horns}

Now that we have defined inner and outer coface maps, we can describe faces and boundaries of dendroidal sets. 

\begin{definition} \cite[pg 16]{mw}  Let $\alpha: T \rightarrow S$ be an (inner or outer) coface map in $\varOmega$. Then the $\alpha$-face of $\varOmega[T]$ is the image of the induced map $\alpha^*:\varOmega[S]\to \varOmega[T]$. We will write $\boundary_\alpha[T]$ for the $\alpha$-face of $\varOmega[T]$. 
\end{definition}

\begin{definition}
The \emph{boundary} of $\varOmega[T]$ is the union over all the faces $\boundary[T]=\bigcup_\alpha \boundary_\alpha[T]$. If we omit the $\beta$-face, we have the $\beta$-horn $\varLambda^\beta[T]=\bigcup_{\alpha\ne \beta}\boundary_\alpha[T]$. If, moreover, $\beta$ is the image of an inner coface map then $ \varLambda^\beta[T]$ is called an \emph{inner horn}.
\end{definition}

A quasi-operad is now defined as a dendroidal set satisfying an inner Kan lifting property. 
\begin{definition}\cite[pg 352]{mw2}\label{quasi-operad definition}
A dendroidal set $X$ is a \emph{quasi-operad}  if for every diagram given by the solid arrows admits a lift 
\[
\xymatrix{
	\varLambda^\beta[T]\ar[r]\ar[d]&X\\
	\varOmega[T]\ar@{.>}[ur]
}
\]  where $T$ ranges over all trees and $\beta$ ranges over all inner coface maps. 
\end{definition}

\begin{definition}\cite[Proposition 1.5]{cm-ho}
A monomorphism of dendroidal sets $X\rightarrow Y$ is said to be \emph{normal} if and
only if for any tree $T$, the action of $\Aut(T)$ on $Y_T\setminus X_T$ is free.
\end{definition}

In analogy to the Joyal model structure on $\sSet$ for quasi-categories (see \cite{juliesurvey} for references), we have the following.

\begin{theorem}\label{infty_operad}\cite[Theorem 2.4]{cm-ho}
There is a model category structure on $\Set^{\varOmega^{op}}$ such that the quasi-operads are the fibrant objects and the normal monomorphisms are the cofibrations.
\end{theorem}

\section{Colored properads, graphical sets, and quasi-properads}\label{Lecture2}

In the previous section we gave a very quick introduction to the dendroidal category using some of the formal language of graph substitution. We will now extend this language to a larger class of graphs to describe properads.  

\emph{Isomorphisms} between graphs preserve all the structure (including orderings) and \emph{weak isomorphisms} between graphs preserve all the structure except the ordering. We denote the category of graphs up to strict isomorphism as $\mathsf{Graph}$. The category $\Graph(m,n)$ is a subcategory of $\Graph$ whose objects are graphs $G$ where $\left|\inn(G)\right|=m$ and $\left|\out(G)\right|=n$. The category $\Graph(\ul{c},\ul{d})$ similarly consists of all $\cCc$-colored graphs with $\inn(G)=\ul{c}=(c_1,..,c_m)$ and $\out(G)=\ul{d}=(d_1,...,d_m)$.

\subsection{Properads} 
Like an operad, a colored properad is a generalization of a category. We have a set of objects, called colors, and now we allow our morphisms to have finite lists of inputs and finite lists of outputs. When we write down a visual representation of a morphism 
$\xymatrix{
(x_1,x_2) \ar[r]^-{f} & (y_1,y_2,y_3)
}$ in a properad we usually write a colored graph 

 \[
\begin{tikzpicture} [scale=.7]
\node[empty](k) at (-2,2){};
\node[empty](j) at (2,2){}; 
\node[shape=circle,draw=black] (f) at (0,0) {f};
\node[empty](v) at (-2,-3) {};
\node[empty](w) at (0,-3) {};
\node[empty](z) at (2,-3) {};
\draw[-] (f) to node [below, fill=white] {$y_1$} (v);
\draw[-] (f) to node [below, fill=white] {$y_2$} (w);
\draw[-] (f) to node [below, fill=white] {$y_3$} (z);
\draw[-] (k) to node [above, fill=white] {$x_1$} (f);
\draw[-] (j) to node [above, fill=white] {$x_2$} (f);

\end{tikzpicture}
\] but it really could be any graph with $2$ inputs and $3$ outputs that is colored by the objects of the properad $P$. In other words, a morphism $\xymatrix{
(x_1,x_2) \ar[r]^-{g} & (y_1,y_2,y_3)
}$ in $P$ is a graph $g\in\Graph(x_1,x_2;y_1,y_2,y_3)$. Composition of morphisms follows the same basic principle of operad composition. In an operad you think of the $\circ_{i}$ composition as plugging the root of a tree into the $i^{th}$ leaf of another tree. For properads we want to be able to take any sub-list of outputs of a graph and glue them to appropriately matched sub-list of inputs in another graph.

\begin{definition}\cite[Definition 3.5]{hrybook}
An $\cCc$-colored properad $P$ consists of 
\begin{itemize} 
\item a set $\cCc=\col(P)$ of colors;
\item for each biprofile $(\ul{c};\ul{d}) = (c_1,...,c_m;d_1,...,d_n)$, a set $P(\ul{c};\ul{d})$;
\item for $\sigma \in \varSigma_m$ and $\tau \in \varSigma_n$, maps
\[
 	P(\ul{c};\ul{d}) \to P(\ul{c}\sigma;\tau\ul{d}) = P(c_{\sigma(1)},...,c_{\sigma(m)};d_{\tau^{-1}(1)},...,d_{\tau^{-1}(n)})
\] 
which assemble into a $\varSigma_m^{op} \times \varSigma^n$ action on the collection $\coprod_{|\ul{c}| = m, |\ul{d}|=n} P(\ul{c}; \ul{d})$;
\item for all $c\in\cCc$, a unit $\textrm{id}_{c}\in P(c;c)$;
\item an associative, until and equivariant composition $$\boxtimes_{b'}^{c'}: P(\ul{c};\ul{d}) \otimes P(\ul{a};\ul{b})\rightarrow P(\ul{a}\circ_{a'}\ul{c}; \ul{b}\circ_{b'}\ul{d})$$ where $\ul{a'}$ and $\ul{b'}$ denote some non-empty finite sublist of $\ul{a}$ and $\ul{b}$, respectively. The notation $\ul{a}\circ_{a'}\ul{c}$ denotes identifying some sublist of $\ul{a}$ with the appropriate sublist of $\ul{c}$.

\end{itemize}

A map of colored properads $f:P\to Q$ consists of
\begin{itemize}
\item $f_0:\Coll(P)\to \Coll(Q)$; 
\item $f_1:P(\ul{c};\ul{d})\to Q(f_0\ul{c};f_0\ul{d})$ for all biprofiles $(\ul{c},\ul{d})$ in $\cCc$. 
\end{itemize}

We denote the category of all colored properads and properad maps between them as $\Properad$. 
\end{definition} 

Properadic composition is easiest to write down in terms of graph substitution. In the previous talk we described a formal process called \emph{graph substitution}, which now repeat in the case of graphs. 

\begin{definition}\cite[2.4]{hrybook}
Given a graph $G\in\Graph(\ul{c};\ul{d})$, and a graph $H_v\in \Graph(\inn(v);\out(v))$ so that each $H_v$ is equipped with bijections
\begin{itemize} 
\item $\inn(H_v)\rightarrow \inn(v)$ and 
\item $\out(H_v)\rightarrow \out(v)$ 
\end{itemize} one constructs a new new graph $G(H_v)\in\Graph(\ul{c};\ul{d})$ by formally identifying $H_v$ with $v\in G.$ In this case we say that $G(H_v)$ is obtained from $G$ by substitution. 
\end{definition}

The following is an example of (uncolored) graph substitution.  Let $G$ and $P$ be the graphs below. 

\[
\begin{tikzpicture}[scale=.5]
\node[empty](k) at (-1,2){};
\node[empty](l) at (0,2){};
\node[empty](j) at (1,2){}; 
\node[empty](a) at (1,3){}; 
\node[empty](b) at (3,3){}; 
\node[empty](c) at (3,-3){}; 
\node[empty](d) at (1,-3){};

\node[shape=circle,draw=white] (G) at (-4,0) {$\Graph(5,6)\ni G=$};
\node[shape=circle,draw=black] (w) at (2,0) {w};
\node[shape=circle, draw=black](x) at (0,-1) {x};
\node[empty](o) at (2,-3){};
\node[empty](p) at (-2,-3){};
\node[empty](m) at (-1,-4){};
\node[empty](n) at (1,-4){};

\draw [-] (k) to [bend right =10] (x);
\draw [-] (j) to [bend left=10]  (x);
\draw [-] (l) to (x);

\draw [-] (a) to[bend right=10]  (w);
\draw [-] (b) to[bend left=10]  (w);

\draw [-] (x) to[bend right=10]  (m);
\draw [-] (x) to [bend left=10] (n);
\draw [-] (x) to [bend left=30] (o);
\draw [-] (x) to [bend right=30] (p);
\draw[-] (w) to [bend left = 20] (x);

\draw [-] (w) to[bend left=10]  (c);
\draw [-] (w) to[bend right=10]  (d);

\node[empty](k) at (5,2){};
\node[empty](l) at (6,2){};
\node[empty](j) at (7,2){}; 
\node[shape=circle,draw=white] (P) at (12,0) {$=P\in \Graph(4,4)$};
\node[shape=circle,draw=black] (u) at (6,0) {u};
\node[shape=circle, draw=black](v) at (6,-2) {v};
\node[empty](o) at (7,-2){};
\node[empty](p) at (5,-2){};
\node[empty](m) at (5,-4){};
\node[empty](n) at (7,-4){};
\node[empty](q) at (8, 1){};

\draw [-] (k) to [bend right =10] (u);
\draw [-] (j) to [bend left=10]  (u);
\draw [-] (l) to (u);

\draw [-] (u) to [bend left= 30]  (v);
\draw[-] (u) to [bend right=30] (v);
\draw[-] (u) to  (v);

\draw [-] (v) to[bend right=10]  (m);
\draw [-] (v) to [bend left=10] (n);
\draw [-] (u) to [bend left=30] (o);
\draw [-] (u) to [bend right=30] (p);
\draw [-] (q) to [bend left=30] (v);

\end{tikzpicture}. 
\]
The graph $G(P_x)$ is still a member in the category $\Graph(5,6),$ but now has a additional three internal edges. 

\[
\begin{tikzpicture}[scale=.5]
\node[empty](k) at (-1,2){};
\node[empty](l) at (0,2){};
\node[empty](j) at (1,2){}; 
\node[empty](a) at (1,1){}; 
\node[empty](b) at (3,1){}; 
\node[empty](c) at (3,-3){}; 
\node[empty](d) at (1,-3){};

\node[shape=circle,draw=black] (u) at (0,0) {u};
\node[shape=circle,draw=black] (w) at (2,-1) {w};
\node[shape=circle, draw=black](v) at (0,-2) {v};
\node[empty](o) at (1,-2){};
\node[empty](p) at (-1,-2){};
\node[empty](m) at (-1,-4){};
\node[empty](n) at (1,-4){};

\draw [-] (k) to [bend right =10] (u);
\draw [-] (j) to [bend left=10]  (u);
\draw [-] (l) to (u);

\draw [-] (u) to [bend left= 30]  (v);
\draw[-] (u) to [bend right=30] (v);
\draw[-] (u) to  (v);

\draw [-] (a) to[bend right=10]  (w);
\draw [-] (b) to[bend left=10]  (w);

\draw [-] (v) to[bend right=10]  (m);
\draw [-] (v) to [bend left=10] (n);
\draw [-] (u) to [bend left=30] (o);
\draw [-] (u) to [bend right=30] (p);
\draw[-] (w) to [bend left = 20] (v);

\draw [-] (w) to[bend left=10]  (c);
\draw [-] (w) to[bend right=10]  (d);

\end{tikzpicture}
\]
To see how this might encode composition, notice that if we squish down the $3$ internal edges between the vertex $u$ and $v$ we would have something that captures our description of composition. 

Following this discussion, one would say that a $\cCc$-colored properad $P$ is the object you get if you consider the set $\cCc$ as objects (or colors) and morphisms between objects $P(\ul{c};\ul{d})$ are a set of (possibly decorated) $\cCc$-colored graphs in $\Graph(\ul{c},\ul{d})$. Composition of a $G$-configuration of morphisms is given by graph substitution
\[\gamma^G_P:P[G]=\prod_{\Vt(G)}P(\inn(v);\out(v))\to P(\inn G;\out G)
\]
where we are ranging over all maps that arise from graph substitution and look like $G(P_x)\rightarrow G$ in the example above. Properadic composition defined in this way is associative and unital because graph substitution is associative and unital~\cite[2.2.4]{hrybook}.  Symmetric group actions come from weak isomorphisms of graphs and properadic composition is equivariant because graph substitution is an operation which is defined up to weak isomorphism class of graphs. 

\begin{remark}\label{partiall_grafted_gen_comp}
Because graph substitution is associative, we observe that it is possible to define properadic composition one operation at a time. In fact, properadic composition is completely determined by the operations described by partially grafted corollas, $\gamma_P^G$, the graph with just an edge (for identities), and the one vertex graphs (for symmetric group actions).
\end{remark} 

\subsection{The graphical category \texorpdfstring{$\varGamma$}{Γ}}

It should by now be unsurprising to hear that given an uncolored graph $G$ we can freely generate a properad $\varGamma(G)$. 

\begin{definition}\cite[Section 5.1]{hrybook} Given an uncolored graph $G$, the properad $\varGamma(G)$ is a colored properad which has the set $\Ed(G)$ as colors and morphisms are generated by the vertices. 

More explicitly, an operation in $\varGamma(G)(\ul{c};\ul{d})$ is a  \emph{$\hat{G}$-decorated graph}, meaning: 
\begin{itemize} 
\item  a graph $H$ in $\Graph(\ul{c};\ul{d})$ whose edges are colored by edges of $G$;
\item  a function from the vertices of $H$ to the vertices of $G$ which is compatible with the coloring of $H$.
\end{itemize} 
\end{definition} 

\begin{example} \cite[Lemma 5.13]{hrybook}
Given the following graph $G$, 
\[
\begin{tikzpicture}[scale=.5]

\node[empty](l) at (0,3){};

\node[shape=circle,draw=white] (P) at (-2,0) {$G=$};
\node[shape=circle,draw=black] (u) at (0,0) {u};
\node[shape=circle, draw=black](v) at (0,-3) {v};

\node[empty](m) at (0,-5){};

\draw [-] (l) to node [above, fill=white]{1} (u);

\draw [-] (u) to [bend left= 30] node [midway, fill=white] {2} (v);
\draw[-] (u) to [bend right=30] node [midway, left, fill=white] {3} (v);

\draw [-] (v) to node [below, fill=white]{4}  (m);

\end{tikzpicture}
\]
the $\hat{G}$-decorated graph $H$ below is 
an example of a morphism in $\varGamma(G)(1,1;4,4)$.

\[
\begin{tikzpicture}[scale=.5]
\node[empty](l) at (0,3){};
\node[empty](l') at (3,3){}; 

\node[shape=circle,draw=white] (H) at (-2,0) {$H=$};
\node[shape=circle,draw=black] (u) at (0,0) {u};
\node[shape=circle, draw=black](v) at (0,-3) {v};

\node[shape=circle,draw=black] (w) at (3,0) {u};
\node[shape=circle, draw=black](z) at (3,-3) {v};

\node[empty](m) at (0,-5){};
\node[empty](m') at (3,-5){};

\draw [-] (l) to node [above, fill=white]{1} (u);
\draw [-] (l') to node [above, fill=white]{1} (w);

\draw [-] (u) to node [near start] {3} (z);
\draw [-] (w) to node[near end] {3} (v);
\draw[-] (u) to [bend right=30] node [midway, left, fill=white] {2} (v);

\draw[-] (w) to [bend left=30] node [midway,  fill=white] {2} (z);

\draw [-] (v) to node [below, fill=white]{4}  (m);
\draw [-] (z) to node [below, fill=white]{4}  (m');

\end{tikzpicture} 
\] 

\end{example} 

Notice that there are many, many more operations in the properads $\varGamma(G)$ than there were in the operads $\varOmega(T)$ that we discussed in the first lecture. This isn't because we forgot to mention operations in $\varOmega(T)$ but rather because of the following lemma. 

\begin{lemma}\cite[Lemma 5.10]{hrybook} If $G$ is a simply connected graph, then each vertex in $G$ can appear in a morphism in the properad $\varGamma(G)$ at most once.
\end{lemma} 

As we mentioned in Remark~\ref{partiall_grafted_gen_comp}, properadic composition is generated by the composites of partially grafted corollas, the graph with one edge, and one vertex graphs. To see that our definition of $\varGamma(G)$ actually is a properad, it then suffices to check the following lemma. 

\begin{lemma} 
All $\hat{G}$-decorated graphs can be built iteratively using partially grafted corollas. 
\end{lemma} 

The naive guess, based on what we expect from understanding $\varDelta$ and $\varOmega$, would be to define a category $\varGamma$ which has as objects the graphical properads $\varGamma(G)$ and morphisms all properad maps between them. This is, unfortunately, not the appropriate definition of $\varGamma$ as there maps between graphical properads that exhibit idiosyncratic behavior. 

\begin{definition}
A properad morphism $f:\varGamma(G)\to \varGamma(H)$ consists of: 
\begin{itemize} 
\item a function $f_0: \Ed(G)\rightarrow\Ed(H)$ together with
\item  a map $f_1:\Vt(G)\rightarrow\{\Vt(H)\text{-decorated graphs}\}$ such that for every $v\in\Vt(G)$, $f_1(v)$ is an $\hat{H}$-decorated graph in $\Graph(f_0\inn v; f_0\out v)$. 
\end{itemize} 
\end{definition}

\begin{definition}\label{image definition} The image of $f: \varGamma(G)\rightarrow \varGamma(H)$ is $f_0G\{f_1(v)\}_{v\in\Vt(G)}$ which is naturally $\Vt(H)$-decorated. The notation $G\{f_1(v)\}_{v\in\Vt(G)}$ stands for performing iterated graph substitution of $\Vt(H)$-decorated graphs at each vertex $v$ in $G$. 
\end{definition}

Morphisms between graphical properads are very strange, so we will pause here and give an explicit description of the image of a map $f:\varGamma(G)\to \varGamma(H)$.

\begin{example}\label{example no vertex graph}
Suppose that $G= \,\downarrow$ is the graph with no vertices and let $Q$ be a $\cCc$-colored properad. Then a properad map $f:\varGamma(\,\downarrow\,)\rightarrow Q$ is a choice of color $c\in\cCc$. 
\end{example} 

\begin{example}  \label{bad_image}
An example of a morphism of graphical properads that behaves poorly is the following. Suppose $G$ is the graph \[
\begin{tikzpicture}[scale=.5]

\node[shape=circle,draw=white] (P) at (-2,0) {$G=$};
\node[shape=circle,draw=black] (u) at (0,0) {u};
\node[shape=circle, draw=black](v) at (0,-3) {v};

\draw [-] (u) to [bend left= 30] node [midway, fill=white] {2} (v);
\draw[-] (u) to [bend right=30] node [midway, left, fill=white] {1} (v);

\end{tikzpicture}
\]
let $f:\varGamma(G)\rightarrow \varGamma(G)$ be the morphism where $f_0$ is the identity on edge sets and 
\begin{itemize} 

\item $f_{1}(v)$ is the $\hat{G}$ decorated graph 
 \[
\begin{tikzpicture}[scale=.5]

\node[shape=circle,draw=white] (P) at (-2,0) {$f_{1}(v)=$};
\node[empty] (u) at (0,0) {};
\node[shape=circle, draw=black](v) at (0,-3) {v};

\draw [-] (u) to [bend left= 30] node [midway, fill=white] {2} (v);
\draw[-] (u) to [bend right=30] node [midway, left, fill=white] {1} (v);

\end{tikzpicture}
\] and 

\item $f_{1}(u)$ is the $\hat{G}$ decorated graph 

\[
\begin{tikzpicture}[scale=.5]

\node[shape=circle,draw=white] (H) at (-4,0) {$f_{1}(u)=$};
\node[shape=circle,draw=black] (u) at (0,0) {u};
\node[shape=circle, draw=black](v) at (2,-4) {v};

\node[shape=circle,draw=black] (w) at (4,0) {u};

\node[empty](m) at (0,-5){};
\node[empty](m') at (4,-5){};

\draw [-] (u) to [bend left= 15] node [very near end] {2} (m');
\draw [-] (w) to [bend right= 15] node [very near end, left] {1} (m);
\draw[-] (u) to [bend right=30] node [near start, left] {2} (v);

\draw[-] (w) to [bend left=30] node [near start] {1} (v);

\end{tikzpicture}. 
\] 
The image of $f$ in $\varGamma(G)$ is then the $\hat{G}$-decorated graph 

\[
\begin{tikzpicture}[scale=.5]

\node[shape=circle,draw=white] (H) at (-4,0) {$\im(G)=$};
\node[shape=circle,draw=black] (u) at (0,0) {u};
\node[shape=circle, draw=black](v) at (2,-4) {v};
\node[shape=circle, draw=black](x) at (2,-8) {v};

\node[shape=circle,draw=black] (w) at (4,0) {u};

\draw [-] (u) to [bend left= 80] node [near end] {2} (x);
\draw [-] (w) to [bend right= 80] node [near end, left] {1} (x);
\draw[-] (u) to [bend right=40] node [near start, left] {2} (v);

\draw[-] (w) to [bend left=40] node [near start] {1} (v);

\end{tikzpicture}. 
\]

\end{itemize}

\end{example}

As we saw in Example~\ref{bad_image}, properad maps $f:\varGamma(H)\rightarrow\varGamma(G)$ need not have the the property that the image of $H$ is a subgraph of $G$. 
This kind of behavior does not show up in dendroidal sets. In fact, for maps into simply connected graphical properads behaves exactly as we would expect from the dendroidal case. 

\begin{proposition} \cite[Proposition 5.32]{hrybook}
If the target of $f:\varGamma(H)\rightarrow \varGamma(G)$ is simply connected (eg any object of $\varOmega$), then $f$ is uniquely determined by what it does on edges.  
\end{proposition} 

As we will explain in section \ref{Lecture3}, in order for our graphical category to have the proper sense of homotopy theory, we will want to force a property of this kind on the category $\varGamma$. 

\begin{proposition}
If the image of $H$ under $f:\varGamma(H)\rightarrow \varGamma(G)$ is a subgraph of $G$, then $f$ is uniquely determined by what it does on edges. 
\end{proposition}

\begin{definition} 
The \emph{graphical category} $\varGamma$ is the category with objects graphical properads and morphisms the subset of properad maps $f:\varGamma(H)\to \varGamma(G)$ consisting of those $f$ with the property that $\im f$ is a subgraph of $G$. 
\end{definition} 

\begin{definition} 
The category of \emph{graphical sets} is the category of presheaves on $\varGamma$, that is  $\Set^{\varGamma^{\op}}$. 
\end{definition} 

For every graph $G$ an element in the set $X_G$ is called a \emph{graphex} with shape $G$. The plural form of graphex is graphices. The representable objects of shape $G$ are $\varGamma[G]=\varGamma(-,G)$. 

\subsection{The properadic nerve} 

The obvious question to ask at this point is how do we know that by throwing out badly behaved properad maps that we are still looking at a reasonable definition of graphical sets?  The \emph{properadic nerve}~\cite[Definition 7.5]{hrybook} is the functor $$N:\Properad \longrightarrow \Set^{\varGamma^{op}}$$defined by $$ (NP)_G = \Properad(\varGamma(G),P ) $$ for $P$ a properad. A graphex in $(NP)_G$ is really a $P$-decoration of $G$, which consists of a coloring of the edges in $G$ by the colors of $P$ and a decoration of each vertex in $G$ by an element in $P$ with the corresponding profiles.

\begin{proposition} \cite[Proposition 7.39]{hrybook} 
 The properadic nerve $$N:\Properad \longrightarrow \Set^{\varGamma^{op}}$$ is fully faithful. 
\end{proposition} 

This proposition implies that while we have lost some maps in $\varGamma$ we have still enough information so that the entire category $\Properad$ sits inside of $\Set^{\varGamma^{op}}$.

\subsection{Cofaces and codegeneracies} 

As in our first lecture, the coface and codegeneracy maps are given by graph substitutions of various kinds. 
A \emph{codegeneracy} map  $\sigma^v:H\rightarrow H(\downarrow)$ is a map induced by substitution of the graph with one edge $\downarrow$ into $(1,1)$-vertex $v\in\Vt(H)$. This has the effect of deleting a vertex.  Like in $\varOmega$, an \emph{inner coface} map will have the effect of ``blowing up'' the graph between two vertices by an inner substitution of a partially grafted corolla $d^{uv}:G\rightarrow G(P)$. 

\begin{example} As an example of an inner coface map consider the graph substitution we have already seen, 

\[
\begin{tikzpicture}[scale=.5]
\node[empty](k) at (-1,2){};
\node[empty](l) at (0,2){};
\node[empty](j) at (1,2){}; 
\node[empty](a) at (1,3){}; 
\node[empty](b) at (3,3){}; 
\node[empty](c) at (3,-3){}; 
\node[empty](d) at (1,-3){};

\node[shape=circle,draw=white] (G) at (-2,0) {$G=$};
\node[shape=circle,draw=black] (w) at (2,0) {w};
\node[empty](f) at (3,0){};
\node[empty](g) at (5,0){};
\node[shape=circle, draw=black](x) at (0,-1) {x};
\node[empty](o) at (2,-3){};
\node[empty](p) at (-2,-3){};
\node[empty](m) at (-1,-4){};
\node[empty](n) at (1,-4){};

\draw[->] (f) to (g);
\draw [-] (k) to [bend right =10] (x);
\draw [-] (j) to [bend left=10]  (x);
\draw [-] (l) to (x);

\draw [-] (a) to[bend right=10]  (w);
\draw [-] (b) to[bend left=10]  (w);

\draw [-] (x) to[bend right=10]  (m);
\draw [-] (x) to [bend left=10] (n);
\draw [-] (x) to [bend left=30] (o);
\draw [-] (x) to [bend right=30] (p);
\draw[-] (w) to [bend left = 20] (x);

\draw [-] (w) to[bend left=10]  (c);
\draw [-] (w) to[bend right=10]  (d);

\node[empty](k) at (6,2){};
\node[empty](l) at (7,2){};
\node[empty](j) at (8,2){}; 
\node[empty](a) at (9,1){}; 
\node[empty](b) at (11,1){}; 
\node[empty](c) at (11,-3){}; 
\node[empty](d) at (9,-3){};

\node[shape=circle,draw=black] (u) at (7,0) {u};
\node[shape=circle,draw=black] (w) at (10,-1) {w};
\node[shape=circle,draw=white] (GP) at (12,-1) {$=G(P)$};
\node[shape=circle, draw=black](v) at (7,-2) {v};
\node[empty](o) at (9,-2){};
\node[empty](p) at (5,-2){};
\node[empty](m) at (6,-4){};
\node[empty](n) at (8,-4){};

\draw [-] (k) to [bend right =10] (u);
\draw [-] (j) to [bend left=10]  (u);
\draw [-] (l) to (u);

\draw [-] (u) to [bend left= 30]  (v);
\draw[-] (u) to [bend right=30] (v);
\draw[-] (u) to  (v);

\draw [-] (a) to[bend right=10]  (w);
\draw [-] (b) to[bend left=10]  (w);

\draw [-] (v) to[bend right=10]  (m);
\draw [-] (v) to [bend left=10] (n);
\draw [-] (u) to [bend left=30] (o);
\draw [-] (u) to [bend right=30] (p);
\draw[-] (w) to [bend left = 20] (v);

\draw [-] (w) to[bend left=10]  (c);
\draw [-] (w) to[bend right=10]  (d);

\end{tikzpicture}. 
\]
where the partially grafted corolla $P\in\Graph(4,4)$ is pictured below. 
\[
\begin{tikzpicture}[scale=.5]
\node[empty](k) at (-1,2){};
\node[empty](l) at (0,2){};
\node[empty](j) at (1,2){}; 
\node[shape=circle,draw=white] (P) at (-2,0) {$P=$};
\node[shape=circle,draw=black] (u) at (0,0) {u};
\node[shape=circle, draw=black](v) at (0,-2) {v};
\node[empty](o) at (1,-2){};
\node[empty](p) at (-1,-2){};
\node[empty](m) at (-1,-4){};
\node[empty](n) at (1,-4){};
\node[empty](q) at (1, 1){};

\draw [-] (k) to [bend right =10] (u);
\draw [-] (j) to [bend left=10]  (u);
\draw [-] (l) to (u);

\draw [-] (u) to [bend left= 30]  (v);
\draw[-] (u) to [bend right=30] (v);
\draw[-] (u) to  (v);

\draw [-] (v) to[bend right=10]  (m);
\draw [-] (v) to [bend left=10] (n);
\draw [-] (u) to [bend left=30] (o);
\draw [-] (u) to [bend right=30] (p);
\draw [-] (q) to [bend left=30] (v);
\end{tikzpicture}
\]

\end{example} 

\begin{example} When restricted to linear graphs, an inner coface map as above is the same as an inner coface map in the simplicial category $\varDelta$. \cite[Example 6.4]{hrybook} \end{example} 

An \emph{outer coface} map $d^{v}:G\rightarrow P(G)$ is an outer substitution of a graph $G$ into a partially grafted corolla.  In the next section, we will discuss how these maps generate the whole category $\varGamma$ in the sense that all morphisms in $\varGamma$ are compositions of (inner or outer) coface maps, codegeneracies and isomorphisms.   

\begin{definition}  A \emph{face} of a representable $\varGamma[H]$ is given by considering the image of an inner or outer coface map. The boundary of $\varGamma[H]$ is defined as $\boundary[H]=\bigcup_{\alpha}\boundary_\alpha [H]\subset \varGamma[H]$ where $\alpha$ ranges over all inner and outer coface maps.  The $\beta$-\emph{horn} is then defined as $\varLambda^\beta[H]\subset \varGamma[H]=\bigcup_{\beta\ne \alpha}\boundary_\alpha [H]$ where $\beta$ is a coface map.

\end{definition} 

\begin{definition} 
A graphical set $X$ is a \emph{quasi-properad} if, for all inner coface maps $\alpha$ and all $H$ in $\varGamma$, the diagram 
\[
\begin{tikzcd}
\varLambda^\alpha[H]\rar\dar&X\\
\varGamma[H]\arrow[dotted]{ur}
\end{tikzcd}
\] admits a lift. 
\end{definition} 

A model category structure on $\Set^{\varGamma^{op}}$ in which quasi-properads are the fibrant objects is work in progress between the authors and D.\,Yau. 

\section{Generalized Reedy structures and a Segal model}\label{Lecture3}

In the previous section we described the graphical category $\varGamma$ and quasi-properads. For more details on why this is precisely a properad ``up to homotopy'' see the description in \cite[7.2]{hrybook}.  In this section we will describe the Reedy structure of $\varGamma$ and use it as a starting point to construct one model category structure for infinity properads.

\subsection{Generalized Reedy categories}

\begin{definition}\label{reedy definition}\cite[Definition 1.1]{bm_reedy}
A dualizable generalized Reedy structure on a small category $\rR$ consists of two subcategories $\rR^+$ and $\rR^-$ which each contain all objects of $\rR$, together with a degree function $\Ob(\rR)\to \nN$ satisfying:
\begin{enumerate}
\item non-invertible morphisms in $\rR^+$ (respectively $\rR^-$) raise (respectively lower degree). Isomorphisms preserve degree.
\item $\rR^+\cap \rR^-=\Iso(\rR)$
\item Every morphism $f$ factors as $f=gh$ such that $g\in \rR^+$ and $h\in \rR^-$ and this factorization is unique up to isomorphism.
\item If $\theta f=f$ for $\theta\in \Iso(\rR)$ and $f\in \rR^-$ then $\theta$ is an identity.
\item  $f\theta=f$ for $\theta\in \Iso(\rR)$ and $f\in \rR^+$ then $\theta$ is an identity.
\end{enumerate}
\end{definition}

\begin{remark} 
A category $\rR$ that satisfies axioms $(1)-(4)$ is a generalized Reedy category. If, in addition, $\rR$ satisfies axiom $(5)$ then $\rR$ is said to be dualizable, which implies that $\rR^{op}$ is also a generalized Reedy category. 
\end{remark}

A (classical) Reedy category is a generalized Reedy category $\rR$ in which every element of $\Iso(\rR)$ is an identity. Examples of classical Reedy categories include $\varDelta$ and $\varDelta^{\op}$. Examples of generalized Reedy categories include the dendroidal category $\varOmega$, finite sets, pointed finite sets, and the cyclic category $\varLambda$.

The main idea of Reedy categories is that we can think about lifting morphisms from $\rR$ to $\Mm^{\rR}$ by induction on the degree of our objects.  To formalize this idea we introduce the notion of latching and matching objects.  

For any $r\in \rR$, the category $\rR^+(r)$ is defined to be a full subcategory of $\rR^+\downarrow r$ consisting of those maps with target $r$ which are not invertible. Similarly, the category $\rR^-(r)$ is the full subcategory of $r\downarrow \rR^{-}$ consisting of maps $\alpha:r\to s$ which are non-invertible. One can now define the \emph{latching object} $$L_r(X)=\colim_{\alpha\in \rR^+(r)}X_s,$$ for each $X$ in $\Mm^{\rR}$ which comes equipped with a map $L_r(X)\rightarrow X_r$.
Similarly, for each $X\in\Mm^{\rR}$ we define the \emph{matching object} $$\lim_{\alpha\in\rR^-(r)}X_s=M_r(X)$$ which comes equipped with a map $X_r\rightarrow M_r(X)$.

\begin{definition} If $\Mm$ is a cofibrantly generated model category, and $\rR$ is generalized Reedy, we say that a morphism $f:X\to Y$ in $\Mm^\rR$ is: 
\begin{itemize}
\item a Reedy cofibration if $X_r\cup_{L_r X}L_r Y\to Y_r$ is a cofibration in $\Mm^{\Aut (r)}$ for all $r\in\rR$;
\item a Reedy weak equivalence if $X_r\to Y_r$ is a weak equivalence in $\Mm^{\Aut(r)}$ for all $r\in\rR$ ; 
\item a Reedy fibration if $X_r\to M_r X\times_{M_r Y}Y_r$ is a fibration in $\Mm^{\Aut(r)}$ for all $r\in\rR$. 
\end{itemize}
\end{definition} 

\begin{theorem}\cite[Theorem 1.6]{bm_reedy} If $\Mm$ is a cofibrantly generated model category and $\rR$ is a generalized Reedy category then the diagram category $\Mm^\rR$ is a model category with the Reedy fibrations, Reedy cofibrations, and Reedy weak equivalences defined above. 
\end{theorem}

\subsection{The graphical category is generalized Reedy} 

\begin{theorem}\cite[6.4]{hrybook}
The graphical category $\varGamma$ is a dualizable generalized Reedy category.
\end{theorem}

The degree function $d:\Ob(\varGamma)\rightarrow \mathbb{N}$ is defined as $d(G) = \left|\Vt(G)\right|$. The positive maps are then those morphisms in $\varGamma$ which are injective on edge sets.  The negative maps are those $H \to G$ which are surjective on edge sets and which, for every vertex $v \in \Vt(G)$, there is a vertex $\tilde{v}\in\Vt(H)$ so that $f_1(\tilde{v})$
is a corolla containing $v$. An alternate, more illuminating, description is given by the following proposition.

\begin{lemma}\cite[6.65]{hrybook}
\begin{itemize} 
\item A map $f:H\to G$ is in $\varGamma^+$ if we can write it as a composition of isomorphisms and coface maps. 
\item A map $f:H\to G$ is in $\varGamma^-$ if we can write it as a composition of isomorphisms and codegeneracy maps.
\end{itemize} 
\end{lemma} 

The proof of this lemma isn't entirely trivial, but the general idea is that codegeneracy maps decrease degree and satisfy the extra condition; coface maps increase degree and are injective on edges. 

We will not fully prove here that $\varGamma$ is Reedy. However, we can show where the decompositions in the third axiom of definition \ref{reedy definition} come from.

\begin{proposition}\cite[6.68]{hrybook} Every map in $f\in\varGamma$ factors as $f = g\circ h$,
where $h\in\varGamma^{-}$ and $g \in\varGamma^{+}$ and this factorization is unique up to isomorphism.
\end{proposition}

\begin{proof}[sketch of existence]
Given a morphism $f\colon G\to K$ in $\varGamma$ we know that for all $v\in\Vt(G)$, $f_1(v)$ is a subgraph of $K$. 

Let us consider  $T\subset \Vt(G)$, the subset of vertices of $G$ such that $f_1(v)=\downarrow$. We can define a graph $G_1=G\{\downarrow_{w}\}_{w\in\Vt(G)}$ which is the graph obtained by substitution of an edge into each $w\in T$ and a corolla substituted into each additional vertex. There is then a map $G\rightarrow G_1$ which is a composition of codegeneracy maps, one for each $w\in T$.  Next, define a a subgraph $G_2$ of $K$ as $G_2=f_0(G_1)$. In other words, $G_2$ is the subgraph obtained by applying $f_0$ to the edges of $G_1$, which makes sense because for each $w\in T$ the incoming edge and outgoing edge of $w$ will have the same image under $f_0$. There is an isomorphism $G_1\rightarrow G_2$ which is just the changing the names of edges via the assignment given by $f_0$. The vertices of $G_2$ are in bijection with the set $\Vt(G_1)\setminus T$. 

It is now the case that the image of $f$, $\im(f) = G_2\{f_1(u)\}_{u\in\Vt(G)\setminus T}$ where each $f_1(u)$ has at least one vertex.  Summarizing, (ignoring coloring) there exists a factorization:
\[
\begin{tikzcd}
G\arrow{rr}{f}\arrow{drr}\dar&&K\\
G_1\rar
&G_2\rar&
\im(f)\uar
.\end{tikzcd}
\]
This shows the existence of the decomposition.
\end{proof}

\begin{example} Let us turn to an example of how we generate $G_1$ for the example of $f$ below.

\[
\begin{tikzpicture} [scale=.6]
\node[empty](k) at (-1,4){};
\node[empty](l) at (0,4){};
\node[empty](j) at (1,4){}; 
\node[empty](p) at (1,0){}; 
\node[empty](q) at (2,0){}; 
\node[shape=circle,draw=white] (G) at (-2,0) {$G=$};
\node[shape=circle,draw=black] (v) at (0,0) {v};
\node[shape=circle,draw=black] (u) at (0,2) {u};
\node[shape=circle, draw=black](x) at (1,-2) {x};

\node[empty] (a) at (4,0){}; 
\node[empty] (b) at (6,0){}; 
\node[empty](m) at (0,-4){};
\node[empty](m') at (2,-4){};

\draw[->](a) to (b);
\draw[-] (k) to (u);
\draw[-] (l) to (u);
\draw[-] (j) to (u);
\draw[-] (u) to (v);
\draw[-] (v) to (x);
\draw[-](x) to (m);
\draw[-](x) to (m');
\draw[-](p) to (x);
\draw[-](q) to (x);

\node[empty](k) at (7,2){};
\node[empty](l) at (8,2){};
\node[empty](j) at (9,2){}; 
\node[empty](p) at (10,2){}; 
\node[empty](p') at (12,2){}; 
\node[shape=circle,draw=white] (G) at (13,0) {$=K$};
\node[shape=circle,draw=black] (w) at (8,0) {w};
\node[shape=circle, draw=black](z) at (9,-2) {z};
\node[shape=circle, draw=black](t) at (11,0) {t};

\node[empty](m) at (8,-4){};
\node[empty](m') at (10,-4){};

\draw[-] (k) to (w);
\draw[-] (l) to (w);
\draw[-] (j) to (w);

\draw[-] (w) to node [midway, fill=white] {1} (z);
\draw[-](z) to (m);
\draw[-](z) to (m');
\draw[-](t) to  [bend right=20] (z);
\draw[-](t) to  [bend left=20] (z);
\draw[-](p) to (t);
\draw[-](p') to (t);

\end{tikzpicture}
\]
Notice that the vertex $v$ is the only vertex in $G$ which has exactly one input and one output, and is mapped by $f$ to the edge in $K$ we have labeled $1$. It follows then that $G_1=G(\downarrow_{v})$ and looks like 
\[
\begin{tikzpicture} [scale=.5]
\node[empty](k) at (-1,2){};
\node[empty](l) at (0,2){};
\node[empty](j) at (1,2){}; 
\node[empty](p) at (1,0){}; 
\node[empty](q) at (2,0){}; 
\node[shape=circle,draw=white] (G) at (-2,0) {$G_1=$};
\node[shape=circle,draw=black] (u) at (0,0) {u};
\node[shape=circle, draw=black](x) at (1,-2) {x};
\node[empty](m) at (0,-4){};
\node[empty](m') at (2,-4){};

\draw[-] (k) to (u);
\draw[-] (l) to (u);
\draw[-] (j) to (u);
\draw[-] (u) to (x);
\draw[-](x) to (m);
\draw[-](x) to (m');
\draw[-](p) to (x);
\draw[-](q) to (x);
\end{tikzpicture} 
\]
The subgraph $G_2$ is now a relabeling and $\im(f)=G_2(f_1(u),f_1(x))$ where $f_1(u)$ is a corolla and $f_1(x)$ is the appropriate partially grafted corolla. 

\end{example} 

\subsection{A Segal model structure for infinity-properads} 

In this section, we attempt to describe a model structure for infinity-properads.
In preparing these notes, we realized the model structure is more complicated than what we presented in the original lectures, for reasons we outline in remarks \ref{broken_remark} and \ref{no model structure}.

We begin with a description of the Segal condition for a graphical set $X\in \Set^{\varGamma^{op}}$.
For $G \in \varGamma$, there is a natural map
\begin{equation}\label{big segal map}
	X_G \to \prod_{v\in \Vt(G)} X_{C_v}
\end{equation}
by using all of the (iterated outer coface) maps $C_v \to G$.
Of course if there is an edge $e$ between two vertices $v$ and $w$, then the two composites $\downarrow_e\, \overset{i_e}\to C_v \to G$ and $\downarrow_e\, \overset{i_e}\to C_w \to G$ are equal, so \eqref{big segal map} factors through a subspace\footnote{This is not a condition when $X_\downarrow = *$ is a one-point set; in that case, $X_G^1$ is just the product from \eqref{big segal map}.}
\[
	X_G^1 = \lim\limits_{\substack{
C_v \leftarrow \, \downarrow_e \rightarrow C_w \\
\text{$e$ an internal} \\ \text{edge of $G$}
}	} \left(\begin{tikzcd}[column sep=tiny, row sep=tiny]
		X_{C_v} \arrow{dr} & & X_{C_w} \arrow{dl}\\
		& X_\downarrow
	\end{tikzcd}\right)
\]
consisting of those sequences $(x_v)$ so that $i_e^*(x_v) = i_e^*(x_w)$ whenever $e$ is an edge between $v$ and $w$.
The \emph{Segal map} is 
\[
	X_G \xrightarrow{\chi_G} X_G^1 \subseteq \prod_{v\in \Vt(G)} X_{C_v}.
\]
If $X = N(P)$ is the nerve of a properad $P$, then $\chi_G$ is an isomorphism \cite[Lemma 7.38]{hrybook}.
In fact, this property \emph{characterizes} those graphical sets which are isomorphic to the nerve of a properad \cite[Theorem 7.42]{hrybook}.

If we allow ourselves to work with graphical \emph{spaces} instead of just graphical sets, then we can replace the isomorphism condition on the Segal maps by a homotopy condition (this type of idea goes all the way back to Segal \cite{segal}).

\begin{definition}\label{segal condition}
	A graphical space $X\in \sSet^{\varGamma^{op}}$ is said to satisfy the \emph{Segal condition} if the Segal map 
\[
	X_G \xrightarrow{\chi_G} X_G^1 \subseteq \prod_{v\in \Vt(G)} X_{C_v}
\]
is a weak homotopy equivalence of simplicial sets between $X_G$ and $X_G^1$ for each graph $G$.
\end{definition}

As in the classical cases, the Segal condition is not categorically well-behaved.
To study the homotopy theory of graphical spaces satisfying the Segal condition, we will build a model structure which allows us to identify such graphical spaces (or, at least those which possess an additional fibrancy condition).

Since $\varGamma$ is a dualizable Reedy category~\cite[Theorem 6.70]{hrybook}, we know that  $\varGamma^{\op}$ is also generalized Reedy. 
Hence, by Berger and Moerdijk \cite[Theorem 1.6]{bm_reedy}, the diagram category $\sSet^{\varGamma^{\op}}$ admits a generalized Reedy model structure.

\begin{remark}\label{broken_remark}
During the lecture, we stated that we could modify this so that the diagram category $\sSet^{\varGamma^{\op}}_{disc}$ admits a Reedy-type model structure, where the subscript $disc$ means that $X_{\downarrow}$ is discrete as a simplicial set.
Indeed, there is such a model structure: the inclusion functor $\sSet^{\varGamma^{\op}}_{disc} \hookrightarrow \sSet^{\varGamma^{\op}}$ admits a left adjoint given by sending $X$ to the pushout of $\pi_0(\operatorname{sk}_0(X)) \leftarrow \operatorname{sk}_0(X) \rightarrow X$, where the skeleton is taken in the $\varGamma$ direction.
One can then lift the model structure from $\sSet^{\varGamma^{\op}}$ using \cite[11.3.2]{hirschhorn}.
Unfortunately, one of the generating cofibrations is not a monomorphism, hence this model structure on $\sSet^{\varGamma^{\op}}_{disc}$ is not cellular.\end{remark}

The following remark is essentially adapted from the end of  \cite[\S 3.12]{threemodels}.

\begin{remark}\label{no model structure}
There is no model structure on $\sSet^{\varGamma^{\op}}_{disc}$ where weak equivalences are levelwise and cofibrations are monomorphisms, as one can see by attempting to factor $\varGamma[\,\downarrow\,] \amalg \varGamma[\,\downarrow\,] \to \varGamma[\,\downarrow\,]$ as a cofibration followed by an acyclic fibration: \[
	\varGamma[\,\downarrow\,] \amalg \varGamma[\,\downarrow\,] \rightarrowtail X \overset\sim\twoheadrightarrow \varGamma[\,\downarrow\,].\]
Since $\varGamma[\,\downarrow\,]_\downarrow$ is a set of cardinality one, the object $X \in \sSet^{\varGamma^{\op}}_{disc}$ would satisfy $2 \leq |X_\downarrow| = 1$.
\end{remark}
 
\begin{definition}[Segal core inclusions] \cite[Definition 7.35]{hrybook}
Given a graph $G$ with at least one vertex let $C_v$ denote the corolla at each $v\in\Vt(G)$ and let $\varGamma[C_v]$ denote the representable graphical set on $C_v$. Define the \emph{Segal core} $\Sc[G]$ as the graphical subset
\[
\Sc[G] 
= \bigcup_{v \in \Vt(G)} \textrm{Im} 
\left(\xymatrix{\varGamma[C_v] \ar[r]^-{i_v} & \varGamma[G]}\right) \subseteq \varGamma[G]
\] 
where $i_{v}$ is an iterated outer coface map. 
Denote by $\xymatrix{\Sc[G] \ar[r]^-{c} & \varGamma[G]}$ the \emph{Segal core inclusion}.
\end{definition}

The reader should compare this definition with \cite[Definition 2.2]{cm-ds}. 
Notice how suggestive this is in light of definition \ref{segal condition}: the map $\chi_G$ is exactly $c^* : \map(\varGamma[G], -) \to \map(\Sc[G], -)$ when $X$ is fibrant.
As we saw above, we cannot guarantee the existence of a left Bousfield localization of the non-cellular category $\sSet_{disc}^{\varGamma^{op}}$ at the set of Segal core inclusions.
Despite that, we still expect that the following holds.

\begin{conjecture}\label{model structure existence}
There is a model structure on $\sSet_{disc}^{\varGamma^{op}}$ analogous to those given in \cite[8.13]{cm-ds} and \cite[5.1]{threemodels}.
\end{conjecture}

In \cite{hry15}, D.\,Yau and the authors gave a model structure on the category $\operatorname{sProperad}$ of simplicially-enriched properads.
The properadic nerve functor that we discussed earlier extends to a functor
\[ N:\operatorname{sProperad}\rightarrow \sSet_{disc}^{\varGamma^{\op}}\]
since $N(P)_\downarrow$ is the \emph{set} of colors of the simplicially-enriched properad $P$.
One should compare the conjectural model structure on $\sSet_{disc}^{\varGamma^{op}}$ with the model structure on $\operatorname{sProperad}$.

\begin{conjecture} 
The properadic nerve functor from simplicial properads to graphical spaces, 
\[ N:\operatorname{sProperad}\rightarrow \sSet_{disc}^{\varGamma^{\op}}\] is the right adjoint in a Quillen equivalence.
\end{conjecture}

\subsection{A diagrammatic overview} We conclude with a diagram which was provided as a handout at our lectures.
It indicates some interconnectedness of many models of categories, operads, properads, and props.

\begin{center}
\includegraphics{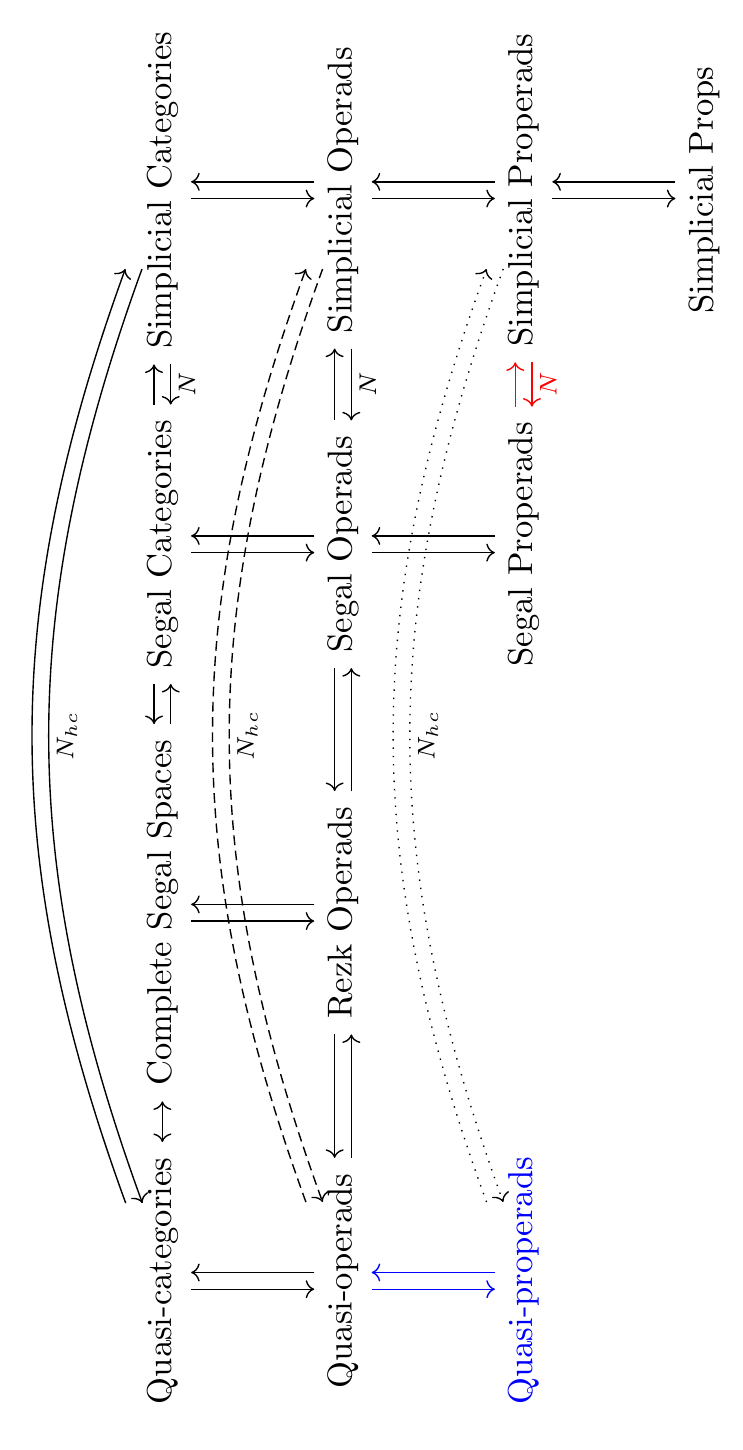}

\end{center}

\noindent The vertical uncolored adjunctions are Quillen adjunctions. The horizontal adjunctions are Quillen equivalences;
\begin{itemize}
	\item precise references for the top row may be found in \cite{juliesurvey}, and
	\item the middle row is contained in \cite{cm-ho,cm-ds,cm-simpop}. 
\end{itemize}
In addition, the model structure for quasi-operads is equivalent to a model structure for Lurie's infinity operads \cite{1305.3658,1606.03826,1302.5756}.
The existence of the model structures in the bottom two slots on the right are \cite{hackneyrobertson2,hry15}.

\bibliographystyle{spmpsci}
\bibliography{talk1}

\begin{thebibliography}{10}
\providecommand{\url}[1]{{#1}}
\providecommand{\urlprefix}{URL }
\expandafter\ifx\csname urlstyle\endcsname\relax
  \providecommand{\doi}[1]{DOI~\discretionary{}{}{}#1}\else
  \providecommand{\doi}{DOI~\discretionary{}{}{}\begingroup
  \urlstyle{rm}\Url}\fi

\bibitem{1302.5756}
Barwick, C.: From operator categories to topological operads (2013).
\newblock \urlprefix\url{https://arxiv.org/abs/1302.5756}

\bibitem{bmresolution}
Berger, C., Moerdijk, I.: Resolution of coloured operads and rectification of
  homotopy algebras.
\newblock In: Categories in algebra, geometry and mathematical physics,
  \emph{Contemp. Math.}, vol. 431, pp. 31--58. Amer. Math. Soc., Providence, RI
  (2007).
\newblock \doi{10.1090/conm/431/08265}.
\newblock \urlprefix\url{http://dx.doi.org/10.1090/conm/431/08265}

\bibitem{bm_reedy}
Berger, C., Moerdijk, I.: On an extension of the notion of {R}eedy category.
\newblock Math. Z. \textbf{269}(3-4), 977--1004 (2011).
\newblock \doi{10.1007/s00209-010-0770-x}.
\newblock \urlprefix\url{http://dx.doi.org/10.1007/s00209-010-0770-x}

\bibitem{threemodels}
Bergner, J.E.: Three models for the homotopy theory of homotopy theories.
\newblock Topology \textbf{46}(4), 397--436 (2007).
\newblock \doi{10.1016/j.top.2007.03.002}.
\newblock \urlprefix\url{http://dx.doi.org/10.1016/j.top.2007.03.002}

\bibitem{juliesurvey}
Bergner, J.E.: A survey of {$(\infty,1)$}-categories.
\newblock In: Towards higher categories, \emph{IMA Vol. Math. Appl.}, vol. 152,
  pp. 69--83. Springer, New York (2010).
\newblock \doi{10.1007/978-1-4419-1524-5_2}.
\newblock \urlprefix\url{http://dx.doi.org/10.1007/978-1-4419-1524-5_2}

\bibitem{bh1}
Bergner, J.E., Hackney, P.: Group actions on {S}egal operads.
\newblock Israel J. Math. \textbf{202}(1), 423--460 (2014).
\newblock \doi{10.1007/s11856-014-1075-2}.
\newblock \urlprefix\url{http://dx.doi.org/10.1007/s11856-014-1075-2}

\bibitem{1606.03826}
Chu, H., Haugseng, R., Heuts, G.: Two models for the homotopy theory of
  $\infty$-operads (2016).
\newblock \urlprefix\url{https://arxiv.org/abs/1606.03826}

\bibitem{cm-ho}
Cisinski, D.C., Moerdijk, I.: Dendroidal sets as models for homotopy operads.
\newblock J. Topol. \textbf{4}(2), 257--299 (2011).
\newblock \doi{10.1112/jtopol/jtq039}.
\newblock \urlprefix\url{http://dx.doi.org/10.1112/jtopol/jtq039}

\bibitem{cm-ds}
Cisinski, D.C., Moerdijk, I.: Dendroidal {S}egal spaces and $\infty$-operads.
\newblock J. Topol. \textbf{6}(3), 675--704 (2013).
\newblock \doi{10.1112/jtopol/jtt004}.
\newblock \urlprefix\url{http://dx.doi.org/10.1112/jtopol/jtt004}

\bibitem{cm-simpop}
Cisinski, D.C., Moerdijk, I.: Dendroidal sets and simplicial operads.
\newblock J. Topol. \textbf{6}(3), 705--756 (2013).
\newblock \doi{10.1112/jtopol/jtt006}.
\newblock \urlprefix\url{http://dx.doi.org/10.1112/jtopol/jtt006}

\bibitem{hackneyrobertson2}
Hackney, P., Robertson, M.: The homotopy theory of props  (2012).
\newblock \urlprefix\url{http://arxiv.org/abs/1209.1087}.
\newblock To appear in Israel J. Math.

\bibitem{hrybook}
Hackney, P., Robertson, M., Yau, D.: Infinity {P}roperads and {I}nfinity
  {W}heeled {P}roperads.
\newblock Lecture Notes in Mathematics, Vol. 2147. Springer (2015).
\newblock Available at \href{http://arxiv.org/abs/1410.6716}{arXiv:1410.6716}
  [math.AT] and
  \href{http://dx.doi.org/10.1007/978-3-319-20547-2}{doi:10.1007/978-3-319-20547-2}

\bibitem{hry15}
Hackney, P., Robertson, M., Yau, D.: A simplicial model for infinity properads
  (2015).
\newblock \urlprefix\url{http://arxiv.org/abs/1502.06522}.
\newblock Preprint

\bibitem{1305.3658}
Heuts, G., Hinich, V., Moerdijk, I.: On the equivalence between {L}urie's model
  and the dendroidal model for infinity-operads (2013).
\newblock \urlprefix\url{https://arxiv.org/abs/1305.3658}

\bibitem{hirschhorn}
Hirschhorn, P.S.: Model {C}ategories and {T}heir {L}ocalizations,
  \emph{Mathematical Surveys and Monographs}, vol.~99.
\newblock American Mathematical Society, Providence, RI (2003)

\bibitem{9807049}
Hirschowitz, A., Simpson, C.: Descente pour les n-champs (descent for n-stacks)
  (1998).
\newblock \urlprefix\url{https://arxiv.org/abs/math/9807049}

\bibitem{kock}
Kock, J.: Graphs, hypergraphs, and properads.
\newblock Collect. Math. \textbf{67}(2), 155--190 (2016).
\newblock \doi{10.1007/s13348-015-0160-0}.
\newblock \urlprefix\url{http://dx.doi.org/10.1007/s13348-015-0160-0}

\bibitem{higheralgebra}
Lurie, J.: Higher {A}lgebra.
\newblock \urlprefix\url{www.math.harvard.edu/~lurie/papers/higheralgebra.pdf}.
\newblock Preprint (May 18, 2011)

\bibitem{moerdijklecture}
Moerdijk, I.: Lectures on dendroidal sets.
\newblock In: Simplicial methods for operads and algebraic geometry, Adv.
  Courses Math. CRM Barcelona, pp. 1--118. Birkh\"auser/Springer Basel AG,
  Basel (2010).
\newblock \doi{10.1007/978-3-0348-0052-5}.
\newblock \urlprefix\url{http://dx.doi.org/10.1007/978-3-0348-0052-5}.
\newblock Notes written by Javier J. Guti{\'e}rrez

\bibitem{mw}
Moerdijk, I., Weiss, I.: Dendroidal sets.
\newblock Algebr. Geom. Topol. \textbf{7}, 1441--1470 (2007).
\newblock \doi{10.2140/agt.2007.7.1441}.
\newblock \urlprefix\url{http://dx.doi.org/10.2140/agt.2007.7.1441}

\bibitem{mw2}
Moerdijk, I., Weiss, I.: On inner {K}an complexes in the category of dendroidal
  sets.
\newblock Adv. Math. \textbf{221}(2), 343--389 (2009).
\newblock \doi{10.1016/j.aim.2008.12.015}.
\newblock \urlprefix\url{http://dx.doi.org/10.1016/j.aim.2008.12.015}

\bibitem{segal}
Segal, G.: Categories and cohomology theories.
\newblock Topology \textbf{13}, 293--312 (1974)

\bibitem{vallette}
Vallette, B.: A {K}oszul duality for {PROP}s.
\newblock Trans. Amer. Math. Soc. \textbf{359}(10), 4865--4943 (2007).
\newblock \doi{10.1090/S0002-9947-07-04182-7}.
\newblock \urlprefix\url{http://dx.doi.org/10.1090/S0002-9947-07-04182-7}

\bibitem{yauoperad}
Yau, D.: Colored operads, \emph{Graduate Studies in Mathematics}, vol. 170.
\newblock American Mathematical Society (2016)

\bibitem{yj}
Yau, D., Johnson, M.W.: A {F}oundation for {P}{R}{O}{P}s, {A}lgebras, and
  {M}odules.
\newblock Math. Surveys and Monographs, Vol. 203. Amer. Math. Soc., Providence,
  RI (2015)

\end{thebibliography}

\end{document}